\documentclass[12pt]{article}

\usepackage{amsthm, amsmath, amssymb}

\numberwithin{equation}{section}

\theoremstyle{plain}	     
\newtheorem{thm}{Theorem}[section] 
\newtheorem{cor}[thm]{Corollary}
\newtheorem{lem}[thm]{Lemma}

\theoremstyle{definition}

\theoremstyle{remark} 
\newtheorem{rem}[thm]{Remark}
	
\newcommand{\vp}{\varphi}

\newcommand{\sn}{\operatorname{sn}}

\newcommand{\am}{\operatorname{am}}
\newcommand{\slem}{\operatorname{sl}}

\begin{document}
\title{Multiple-angle formulas of generalized trigonometric functions
with two parameters
\footnote{This work was supported by JSPS KAKENHI Grant Number 24540218.}}
\author{Shingo Takeuchi \\
Department of Mathematical Sciences\\
Shibaura Institute of Technology
\thanks{307 Fukasaku, Minuma-ku,
Saitama-shi, Saitama 337-8570, Japan. \endgraf
{\it E-mail address\/}: shingo@shibaura-it.ac.jp \endgraf
}}
\date{}

\maketitle

\begin{abstract}
Generalized trigonometric functions with two parameters were introduced 
by Dr\'{a}bek and Man\'{a}sevich to study an inhomogeneous eigenvalue
problem of the $p$-Laplacian. Concerning these functions,
no multiple-angle formula 
has been known except for the classical cases and 
a special case discovered by Edmunds, Gurka and Lang,
not to mention addition theorems.
In this paper, we will present 
new multiple-angle formulas which are established 
between two kinds of the generalized trigonometric functions,
and apply the formulas to generalize classical topics related to 
the trigonometric functions and the lemniscate function. 
\end{abstract}



\section{Introduction}

Let $p,\ q \in (1,\infty)$ be any constants. 
We define $\sin_{p,q}{x}$ by
the inverse function of 
$$\sin_{p,q}^{-1}{x}:=\int_0^x \frac{dt}{(1-t^q)^{1/p}}, \quad 0 \leq x \leq 1,$$
and
\begin{equation}
\label{eq:pipq}
\pi_{p,q}:=2\sin_{p,q}^{-1}{1}=2\int_0^1 \frac{dt}{(1-t^q)^{1/p}}
=\frac2q B\left(\frac{1}{p^*},\frac1q\right),
\end{equation}
where $p^*:=p/(p-1)$ and $B$ denotes the beta function.
The function $\sin_{p,q}{x}$ is increasing in $[0,\pi_{p,q}/2]$ onto $[0,1]$.
We extend it to $(\pi_{p,q}/2,\pi_{p,q}]$ by $\sin_{p,q}{(\pi_{p,q}-x)}$
and to the whole real line $\mathbb{R}$ as the odd $2\pi_{p,q}$-periodic 
continuation of the function. 
Since $\sin_{p,q}{x} \in C^1(\mathbb{R})$,
we also define $\cos_{p,q}{x}$ by $\cos_{p,q}{x}:=(\sin_{p,q}{x})'$.
Then, it follows that 
$$|\cos_{p,q}{x}|^p+|\sin_{p,q}{x}|^q=1.$$
In case $p=q=2$, it is obvious that $\sin_{p,q}{x},\ \cos_{p,q}{x}$ 
and $\pi_{p,q}$ are reduced to the ordinary $\sin{x},\ \cos{x}$ and $\pi$,
respectively. 
This is a reason why these functions and the constant are called
\textit{generalized trigonometric functions} (with parameter $(p,q)$)
and \textit{the generalized $\pi$}, respectively. 

Dr\'{a}bek and Man\'{a}sevich \cite{DM} introduced 
the generalized trigonometric functions with two parameters
to study an inhomogeneous eigenvalue problem of $p$-Laplacian.
They gave a closed form of 
solutions $(\lambda,u)$ of the eigenvalue problem
$$-(|u'|^{p-2}u')'=\lambda |u|^{q-2}u, \quad u(0)=u(L)=0.$$
Indeed, for any $n=1,2,\ldots$, there exists a curve of 
solutions $(\lambda_{n,R},u_{n,R})$
with a parameter $R \in \mathbb{R} \setminus \{0\}$ such that
\begin{gather}
\lambda_{n,R}=\frac{q}{p^*}\left(\frac{n\pi_{p,q}}{L}\right)^p |R|^{p-q},
\label{eq:ev}\\
u_{n,R}(x)=R\sin_{p,q}\left(\frac{n\pi_{p,q}}{L}x\right)
\label{eq:ef}
\end{gather}
(see also \cite{T}). Conversely, there exists no other solution of the 
eigenvalue problem. Thus, the generalized trigonometric functions 
play important roles to study problems of the $p$-Laplacian.

It is of interest to know whether the generalized trigonometric functions
have multiple-angle formulas unless $p=q=2$.
A few multiple-angle formulas seem to be known.
Actually, in case $2p=q=4$, the function $\sin_{p,q}{x}=\sin_{2,4}{x}$
coincides with the lemniscate sine function $\slem{x}$,
whose inverse function is defined as
$$\slem^{-1}{x}:=\int_0^x \frac{dt}{\sqrt{1-t^4}}.$$
Furthermore, $\pi_{2,4}$ is equal to the lemniscate constant 
$\varpi:=2\slem^{-1}{1}=2.6220\cdots$.
Concerning $\slem{x}$ and $\varpi$, 
we refer to the reader to \cite[p.\,81]{PS}, \cite{To} and \cite[\S{22.8}]{WW}.
Since $\slem{x}$ has the multiple-angle formula
\begin{equation}
\label{eq:maflem}
\slem{(2x)}=\frac{2\slem{x}\sqrt{1-\slem^4{x}}}
{1+\slem^4{x}},\quad 0 \leq x \leq \frac{\varpi}{2},
\end{equation}
we see that
$$\sin_{2,4}{(2x)}=\frac{2\sin_{2,4}{x}\cos_{2,4}{x}}
{1+\sin_{2,4}^4{x}},\quad 0 \leq x \leq \frac{\pi_{2,4}}{2}.$$
Also in case $p^*=q=4$, it is possible to show that 
$\sin_{p,q}{x}=\sin_{4/3,4}{x}$ can be expressed in terms of 
the Jacobian elliptic function,
whose multiple-angle formula yields   
\begin{equation}
\label{eq:EGL}
\sin_{4/3,4}{(2x)}=\frac{2\sin_{4/3,4}{x}\cos^{1/3}_{4/3,4}{x}}
{\sqrt{1+4\sin_{4/3,4}^4{x}\cos_{4/3,4}^{4/3}{x}}}
\quad 0 \leq x <\frac{\pi_{4/3,4}}{4}.
\end{equation}
The formula \eqref{eq:EGL} was investigated by Edmunds, Gurka 
and Lang \cite[Proposition 3.4]{EGL}.
They also proved an addition theorem for $\sin_{4/3,4}{x}$ 
involving \eqref{eq:EGL}.

In this paper, we will present multiple-angle formulas
which are established between 
two kinds of the generalized trigonometric functions 
with parameters $(2,p)$ and $(p^*,p)$.

\begin{thm}
\label{thm:main}
For $p \in (1,\infty)$ and $x \in [0,2^{-2/p}\pi_{2,p}]
=[0,\pi_{p^*,p}/2]$, we have
\begin{equation}
\label{eq:main1}
\sin_{2,p}{(2^{2/p} x)}=2^{2/p} \sin_{p^*,p}x \cos_{p^*,p}^{p^*-1}x
\end{equation}
and
\begin{align}
\cos_{2,p}{(2^{2/p} x)}&=\cos_{p^*,p}^{p^*}{x}-\sin_{p^*,p}^p{x} \notag \\
&=1-2\sin_{p^*,p}^p{x}=2\cos_{p^*,p}^{p^*}{x}-1. \label{eq:main2}
\end{align}
Moreover, for $x \in \mathbb{R}$, we have
\begin{equation}
\label{eq:formula2} 
\sin_{2,p}{(2^{2/p} x)}=2^{2/p} \sin_{p^*,p}{x}|\cos_{p^*,p}{x}|^{p^*-2}\cos_{p^*,p}{x}
\end{equation}
and
\begin{align}
\cos_{2,p}{(2^{2/p} x)}
&=|\cos_{p^*,p}{x}|^{p^*}-|\sin_{p^*,p}{x}|^p \notag \\
&=1-2|\sin_{p^*,p}{x}|^p=2|\cos_{p^*,p}{x}|^{p^*}-1.\label{eq:formula}
\end{align}
\end{thm}



In Theorem \ref{thm:main}, the fact
\begin{equation}
\label{eq:key}
\frac{\pi_{2,p}}{2^{2/p}}=\frac{\pi_{p^*,p}}{2}
\end{equation}
is the special case $n=2$ of the following identity.

\begin{thm}
\label{cor:pi}
Let $2 \leq n<p+1$. Then
$$\pi_{\frac{p}{p-1},p} \pi_{\frac{p}{p-2},p} \cdots \pi_{\frac{p}{p-n+1},p}
=n^{1-n/p}\pi_{\frac{n}{n-1},p}\pi_{\frac{n}{n-2},p}\cdots \pi_{\frac{n}{1},p}.$$
\end{thm}

We give a series expansion of $\pi_{p^*,p}$
as a counterpart of the Gregory-Leibniz series for $\pi$.
It is worth pointing out that $\pi_{p^*,p}$ is the area 
enclosed by the $p$-circle $|x|^p+|y|^p=1$
(see \cite{LE}).

\begin{thm}
\label{prop:gls}
\begin{align*}
\frac{\pi_{p^*,p}}{4}
&=\sum_{n=0}^\infty\frac{(2/p)_n}{n!}\frac{(-1)^n}{pn+1} \\
&=1-\frac{2}{p(p+1)}+\frac{2+p}{p^2(2p+1)}-\frac{(2+p)(2+2p)}{3p^3(3p+1)}+\cdots,
\end{align*}
where $(a)_n:=\Gamma(a+n)/\Gamma(a)=a(a+1)(a+2)\cdots(a+n-1)$
and $\Gamma$ denotes the gamma function.
\end{thm}

We will apply Theorems \ref{thm:main}--\ref{prop:gls} 
to the following problems (I)--(V).

(I) \textit{An alternative proof of \eqref{eq:EGL}}.
It should be noted that 
the multiple-angle formula \eqref{eq:main1} in Theorem \ref{thm:main}
allows \eqref{eq:EGL}
to be rewritten in terms of the lemniscate function $\slem{x}=\sin_{2,4}{x}$:
\begin{equation*}
\sin_{4/3,4}{(2x)}=\frac{\sqrt{2}\slem{(\sqrt{2}x)}}
{\sqrt{1+\slem^4{(\sqrt{2}x)}}},
\quad 0 \leq x <\frac{\pi_{4/3,4}}{4}=\frac{\varpi}{2\sqrt{2}},
\end{equation*}
where the last equality above follows from \eqref{eq:key} with $\pi_{2,4}=\varpi$.
This indicates that it is possible to obtain \eqref{eq:EGL} from
the multiple-angle formula \eqref{eq:maflem} for the lemniscate function.

(II) \textit{A relation between eigenvalue problems of the $p$-Laplacian
and that of the Laplacian}.
Let $u$ be a function with $(n-1)$-zeros in $(0,L)$ 
satisfying 
$$-(|u'|^{p-2}u')'=\lambda |u|^{p^*-2}u, \quad u(0)=u(L)=0$$
for some $\lambda>0$. 
Similarly, let $v$ be a function with $n$-zeros in $(0,L)$
satisfying 
$$-(|v'|^{p-2}v')'=\mu |v|^{p^*-2}v, \quad v'(0)=v'(L)=0$$
for some $\mu>0$. 
Then, by Theorem \ref{thm:main}, 
we can show that the product $w=uv$ is a function 
with $(2n-1)$-zeros in $(0,L)$ satisfying
$$-w''=2p^*(\lambda \mu)^{1/p} |w|^{p^*-2}w,\quad w(0)=w(L)=0.$$
The curious fact is the consequence of a straightforward calculation with 
\eqref{eq:ev},\ \eqref{eq:ef},\ \eqref{eq:formula2} and \eqref{eq:key}.
Such a relation between the eigenvalue problems of the $p$-Laplacian
and that of the Laplacian may be known.
However, we can not find a literature proving it,
while the assertion in case $p=2$ is trivial because
$$w=\sin{\left(\frac{n\pi}{L}x\right)}\cos{\left(\frac{n\pi}{L}x\right)}
=\frac12 \sin{\left(\frac{2n\pi}{L}x\right)}.$$

(III) \textit{A pendulum-type equation with the $p$-Laplacian}.
We give a closed form of 
solutions of the pendulum-type equation
$$-(|\theta'|^{p-2}\theta')'
=\lambda^p |\sin_{2,p}{\theta}|^{p-2}\sin_{2,p}{\theta}.$$
In case $p=2$, this equation is the ordinary pendulum equation
$-\theta''=\lambda^2 \sin{\theta}$ and it is well known that the solutions
can be expressed in terms of 
the Jacobian elliptic function. 
We will obtain an expression of the solution for 
the pendulum-type equation above by using our special functions 
involving a generalization of the Jacobian elliptic function in \cite{T,T2}.
There are studies of other (forced) pendulum-type equations with $p$-Laplacian
versus $\sin{\theta}$ in \cite{M}; versus $\sin_{p,p}{\theta}$ in \cite{AMN}, 
for the purpose of finding periodic solutions.

(IV) \textit{Catalan-type constants}.
Catalan's constant, which occasionally appears in estimates in combinatorics,
is defined by
$$G=\sum_{n=0}^\infty \frac{(-1)^n}{(2n+1)^2}=0.9159\cdots.$$
We can find a lot of representation of $G$ in \cite{Bra}; 
for a typical example,
\begin{equation}
\label{eq:Gs} 
\frac12\int_0^{\pi/2} \frac{x}{\sin{x}}\,dx=G.
\end{equation}
The multiple-angle formula \eqref{eq:main1} gives a generalization of \eqref{eq:Gs} as
\begin{equation}
\label{eq:Gp1}
\frac{1}{2^{2/p}}\int_0^{\pi_{2,p}/2} \frac{x}{\sin_{2,p}{x}}\,dx
=\sum_{n=0}^\infty \frac{(2/p)_n}{n!} \frac{(-1)^n}{(pn+1)^2}.
\end{equation}
In case $p=2$, 
the formula \eqref{eq:Gp1} coincides 
with \eqref{eq:Gs}. Moreover, for $p=4$ we obtain the interesting formula
$$\frac{1}{\sqrt{2}} \int_0^{\varpi/2} \frac{x}{\slem{x}}\,dx
=\sum_{n=0}^\infty \frac{(1/2)_n}{n!} \frac{(-1)^n}{(4n+1)^2}.$$

(V) \textit{Series expansions of the lemniscate constant $\varpi$}.
The lemniscate constant $\varpi$ has the formula (\cite[Theorem 5]{To}):
\begin{equation*}
\frac{\varpi}{2}
=1+\frac{1}{10}+\frac{1}{24}+\frac{5}{208}+\cdots
+\frac{(2n-1)!!}{(2n)!!}\frac{1}{4n+1}+\cdots,
\end{equation*}
where $(-1)!!:=1$. 
For this, using Theorem \ref{prop:gls} with \eqref{eq:key}, we will obtain 
$$\frac{\varpi}{2\sqrt{2}}
=1-\frac{1}{10}+\frac{1}{24}-\frac{5}{208}+\cdots
+\frac{(2n-1)!!}{(2n)!!}\frac{(-1)^n}{4n+1}+\cdots,$$
which does not appear in Todd \cite{To} and seems to be new. 
We will also produce some other formulas of $\varpi$. 

This paper is organized as follows.
Section 2 is devoted to the proofs of Theorems \ref{thm:main}--\ref{prop:gls}.
In Section 3, we deal with the above-mentioned problems (I)--(V).


\section{The multiple-angle formulas}
\label{sec:def}

Let $p,\ q \in (1,\infty)$ and $x \in (0,\pi_{p,q}/2)$.
It is easy to see that
\begin{gather*}
\cos_{p,q}^p{x}+\sin_{p,q}^q{x}=1, \\
(\sin_{p,q}{x})'=\cos_{p,q}{x}, \quad
(\cos_{p,q}{x})'=-\frac{q}{p}\sin_{p,q}^{q-1}{x}\cos_{p,q}^{2-p}{x}, \\
(\cos_{p,q}^{p-1}{x})'=-\frac{q}{p^*}\sin_{p,q}^{q-1}{x}. 
\end{gather*}
If we extend to these formulas for any $x \in \mathbb{R}$, then 
the last one, for example, corresponds to 
\begin{equation}
(|\cos_{p,q}{x}|^{p-2}\cos_{p,q}{x})'
=-\frac{q}{p^*}|\sin_{p,q}{x}|^{q-2}\sin_{p,q}{x}. \label{eq:diffsin} 
\end{equation}

In a particular case, 
\begin{gather}
\cos_{p^*,p}^{p^*}{x}+\sin_{p^*,p}^p{x}=1, \label{eq:pyta} \\
(\sin_{p^*,p}{x})'=\cos_{p^*,p}{x}, \quad
(\cos_{p^*,p}{x})'=-(p-1)\sin_{p^*,p}^{p-1}{x}\cos_{p^*,p}^{2-p^*}{x}, \notag \\
(\cos_{p^*,p}^{p^*-1}{x})'=-\sin_{p^*,p}^{p-1}{x}. \notag
\end{gather}
From the last one and the differentiation of inverse functions,
$$(\cos_{p^*,p}^{p^*-1})^{-1}{(y)}=\int_y^1 \frac{dt}{(1-t^p)^{1/p^*}}, \quad 0 \leq y \leq 1,$$
hence
$$\sin_{p^*,p}^{-1}{y}+(\cos_{p^*,p}^{p^*-1})^{-1}{(y)}=\frac{\pi_{p^*,p}}{2}.$$
Therefore, for $x \in [0,\pi_{p^*,p}/2]$
\begin{gather}
\sin_{p^*,p}\left(\frac{\pi_{p^*,p}}{2}-x\right)=\cos_{p^*,p}^{p^*-1}{x}, \label{eq:sym} \\
\cos_{p^*,p}^{p^*-1}\left(\frac{\pi_{p^*,p}}{2}-x\right)=\sin_{p^*,p}{x}. 
\label{eq:sym2}
\end{gather}

Throughout this paper, the following function is useful:
$$\tau_{p,q}(x):=\frac{\sin_{p,q}{x}}{|\cos_{p,q}{x}|^{q/p-1}\cos_{p,q}{x}}, \quad 
x \neq \frac{2n+1}{2}\pi_{p,q},\ n \in \mathbb{Z}.$$
Then, it follows immediately from \eqref{eq:sym} and \eqref{eq:pyta} that 

\begin{lem}
\label{lem:pi4}
For $x \in (0,\pi_{p^*,p}/2)$, $\tau_{p^*,p}{(x)}=1$ implies 
$x=\pi_{p^*,p}/4$. Moreover, 
$\sin_{p^*,p}^{-1}{(2^{-1/p})}=\cos_{p^*,p}^{-1}{(2^{-1/p^*})}
=\pi_{p^*,p}/4$.
\end{lem}

%
%
%

Let us prove the multiple-angle formulas in Theorem \ref{thm:main}.

\begin{proof}[Proof of Theorem \ref{thm:main}]
Let $x \in [0,\pi_{p^*,p}/4]$. Then, $y=\sin_{p^*,p}{x} \in [0,2^{-1/p}]$ by Lemma \ref{lem:pi4}.
Setting $t^p=(1-(1-s^p)^{1/2})/2$ in
$$\sin_{p^*,p}^{-1}{y}=\int_0^y \frac{dt}{(1-t^p)^{1/p^*}},$$
we have
\begin{align*}
\sin_{p^*,p}^{-1}{y}
&=\int_0^{y(4(1-y^p))^{1/p}}
\dfrac{\dfrac{2^{-1-1/p}s^{p-1}}{(1-s^p)^{1/2}(1-(1-s^p)^{1/2})^{1-1/p}}\,ds}
{2^{-1+1/p} (1+(1-s^p)^{1/2})^{1-1/p}}\\
&=2^{-2/p}
\int_0^{y(4(1-y^p))^{1/p}} 
\frac{ds}{(1-s^p)^{1/2}};
\end{align*}
that is, 
\begin{equation}
\label{eq:pe}
\sin_{p^*,p}^{-1}{y}=2^{-2/p} \sin_{2,p}^{-1}{(y(4(1-y^p)^{1/p})}.
\end{equation}
Hence we obtain
$$\sin_{2,p}{(2^{2/p} x)}=2^{2/p} \sin_{p^*,p}{x}\cos_{p^*,p}^{p^*-1}{x},$$
and \eqref{eq:main1} is proved.
In particular, letting $y=2^{-1/p}$ in \eqref{eq:pe} and using Lemma \ref{lem:pi4}, 
we get
$$\frac{\pi_{p^*,p}}{4}=2^{-2/p} \sin_{2,p}^{-1}{1}
=\frac{\pi_{2,p}}{2^{1+2/p}},$$
which implies \eqref{eq:key}.

Next, let $x \in (\pi_{p^*,p}/4,\pi_{p^*,p}/2]$
and $y:=\pi_{p^*,p}/2-x \in [0,\pi_{p^*,p}/4)$. 
By the symmetry properties
\eqref{eq:sym} and \eqref{eq:sym2}, we obtain
$$2^{2/p}\sin_{p^*,p}{x}\cos_{p^*,p}^{p^*-1}{x}
=2^{2/p}\cos_{p^*,p}^{p^*-1}{y}\sin_{p^*,p}{y}.$$
According to the argument above, 
the right-hand side is identical to $\sin_{2,p}{(2^{2/p}y)}$. 
Moreover, \eqref{eq:key} gives
$$\sin_{2,p}{(2^{2/p}y)}=\sin_{2,p}{(\pi_{2,p}-2^{2/p}x)}
=\sin_{2,p}{(2^{2/p}x)}.$$

The formula \eqref{eq:main2} is deduced from differentiating both sides
of \eqref{eq:main1}. Moreover, \eqref{eq:formula2} and \eqref{eq:formula}
come from the periodicities of the functions,
\end{proof}

%
\begin{proof}[Proof of Theorem \ref{cor:pi}]
Setting $x=1/n$ and $y=1/p$ in the formula of the beta function
(see \cite[\S{12.15}, Example]{WW})
$$B(nx,ny)=\frac{1}{n^{ny}} \frac{\prod_{k=0}^{n-1} 
B\left(x+k/n,y\right)}{\prod_{k=1}^{n-1}B(ky,y)}, \quad n \geq 2,\ x,\ y>0,$$
we have
$$\frac{p}{n}=\frac{1}{n^{n/p}}
\frac{\prod_{k=1}^nB(k/n,1/p)}{\prod_{k=1}^{n-1}B(k/p,1/p)}.$$
Hence,
$$\prod_{k=1}^{n-1} B\left(\frac{k}{p},\frac1p\right)=n^{1-n/p}
\prod_{k=1}^{n-1}B\left(\frac{k}{n},\frac1p\right), \quad n \geq 2.$$
From \eqref{eq:pipq}, this is rewritten as
$$\prod_{k=1}^{n-1} \pi_{(p/k)^*,p}=n^{1-n/p}
\prod_{k=1}^{n-1}\pi_{(n/k)^*,p}, \quad
2 \leq n <p+1,$$
which is precisely the assertion of the theorem.
\end{proof}

\begin{rem}
Taking $n=2$ in Theorem \ref{cor:pi}, we have the relation \eqref{eq:key}
between $\pi_{p^*,p}$ and $\pi_{2,p}$.
In fact,   
\eqref{eq:key} is equivalent to the duplication formula of the gamma function
(see \cite[\S{12.15}, Corollary]{WW})
$$\Gamma (2x)=\frac{2^{2x-1}}{\sqrt{\pi}} \Gamma(x)\Gamma\left(x+\frac12\right).$$
\end{rem}

\begin{proof}[Proof of Theorem \ref{prop:gls}]
Let $x \in (0,1)$. Differentiating the inverse function of $\tau_{p^*,p}{(x)}$, we have
$$\tau_{p^*,p}^{-1}{(x)}=\int_0^x \frac{dt}{(1+t^p)^{2/p}}.$$
Hence
\begin{align}
\label{eq:tan}
\tau_{p^*,p}^{-1}{(x)}
=\int_0^x \sum_{n=0}^\infty \binom{-2/p}{n} t^{pn}\, dt
=x \sum_{n=0}^\infty \frac{(2/p)_n}{n!}\frac{(-x^p)^n}{pn+1}.
\end{align}
By Abel's continuity theorem \cite[\S3.71]{WW}, 
it is sufficient to show that the right-hand side of \eqref{eq:tan} converges 
at $x=1$; i.e., the series
$$\sum_{n=0}^\infty \frac{(2/p)_n}{n!}\frac{(-1)^n}{pn+1}
=:\sum_{n=0}^\infty (-1)^na_n$$
converges. This is an alternating series and $\{a_n\}$ is decreasing because
$$0 \leq \frac{a_{n+1}}{a_n}=\frac{(2/p+n)(pn+1)}{(n+1)(pn+p+1)}
=\frac{(n+1/p)(pn+2)}{(n+1)(pn+p+1)}<1.$$ 
Moreover, $\{a_n\}$ converges to $0$ as $n \to \infty$. Indeed,
Euler's formula for the gamma function \cite[\S12.11, Example]{WW} gives 
\begin{align*}
\lim_{n \to \infty} a_n
&=\lim_{n \to \infty}
\frac{2/p (2/p+1)(2/p+2)\cdots (2/p+n-1)}{(n-1)!n^{2/p}}
\frac{n^{2/p}}{n(pn+1)}\\
&=\frac{1}{\Gamma(2/p)}\cdot 0=0.
\end{align*}
Therefore, the series above converges to $\tau_{p^*,p}^{-1}(1)$ 
(see for instance \cite[\S 2.31, Corollary (ii)]{WW}). 
From Lemma \ref{lem:pi4}, we concludes the theorem. 
\end{proof}

\begin{rem}
Combining \eqref{eq:formula2} and \eqref{eq:formula},
we can assert that $\tau_{2,p}$ and $\tau_{p^*,p}$ satisfy the multiple-angle formula
$$\tau_{2,p}{(2^{2/p}x)}=\frac{2^{2/p}\tau_{p^*,p}{(x)}}{1-|\tau_{p^*,p}{(x)}|^p},$$
which coincides with that of the tangent function if $p=2$.
\end{rem}


\section{Applications}
\label{sec:appl}


\subsection{An alternative proof of \eqref{eq:EGL}}

Let us give 
an alternative proof of the multiple-angle formula
of $\sin_{4/3,4}{x}$:
\begin{equation}
\sin_{4/3,4}{(2x)}=\frac{2\sin_{4/3,4}{x}\cos_{4/3,4}^{1/3}{x}}
{\sqrt{1+4\sin_{4/3,4}^4{x}\cos_{4/3,4}^{4/3}{x}}},
\quad 0 \leq x <\frac{\pi_{4/3,4}}{4}, \tag{\ref{eq:EGL}}
\end{equation}
which was discovered by Edmunds, Gurka and Lang \cite{EGL}.


Recall that $\sin_{2,4}{x}$ is equal to the
lemniscate function $\slem{x}$.
Applying \eqref{eq:main1} in case $p=4$
with $x$ replaced by $2x \in [0,\pi_{4/3,4}/2)$, we get
\begin{equation}
\label{eq:cubic}
\slem{(2\sqrt{2}x)}=\sqrt{2}\sin_{4/3,4}{(2x)}(1-\sin_{4/3,4}^4{(2x)})^{1/4}.
\end{equation}

First, we consider the case 
$$0 \leq x <\frac{\pi_{4/3,4}}{8}=\frac{\varpi}{4\sqrt{2}}.$$
Then, since $0 \leq 2\sin_{4/3,4}^4{(2x)}<1$ by Lemma \ref{lem:pi4}, 
the equation \eqref{eq:cubic} gives
$$2\sin_{4/3,4}^4{(2x)}
=1-\sqrt{1-\slem^4{(2\sqrt{2}x)}}.$$
Set $S=S(x):=\slem{(\sqrt{2}x)}$. Using the multiple-angle formula \eqref{eq:maflem} 
for the lemniscate function, we have
\begin{align*}
2\sin_{4/3,4}^4{(2x)}
&=1-\sqrt{1-\left(\frac{2S\sqrt{1-S^4}}{1+S^4}\right)^4}\\
&=1-\frac{\sqrt{1-12S^4+38S^8-12S^{12}+S^{16}}}{(1+S^4)^2}\\
&=1-\frac{|1-6S^4+S^8|}{(1+S^4)^2}.
\end{align*}
Since $0 \leq S <\slem{(\varpi/4)}=(3-2\sqrt{2})^{1/4}$, 
evaluated by \eqref{eq:maflem},
we see that $1-6S^4+S^8 \geq 0$. 
Thus, 
\begin{equation}
\label{eq:case1}
2\sin_{4/3,4}^4{(2x)}
=1-\frac{1-6S^4+S^8}{(1+S^4)^2}=\frac{8S^4}{(1+S^4)^2}.
\end{equation}
Therefore, by \eqref{eq:main1}, 
$$\sin_{4/3,4}{(2x)}=\frac{\sqrt{2}S}{\sqrt{1+S^4}}
=\frac{2\sin_{4/3,4}{x}\cos_{4/3,4}^{1/3}{x}}
{\sqrt{1+4\sin_{4/3,4}^4{x}\cos_{4/3,4}^{4/3}{x}}}.$$

In the remaining case 
$$\frac{\varpi}{4\sqrt{2}}=\frac{\pi_{4/3,4}}{8} 
\leq x <\frac{\pi_{4/3,4}}{4}=\frac{\varpi}{2\sqrt{2}},$$
it follows easily that $1 \leq 2\sin_{4/3,4}^4{(2x)}<2$ and $1-6S^4+S^8 \leq 0$,
hence we obtain \eqref{eq:case1} again.
The proof of \eqref{eq:EGL} is complete.

\subsection{A relation between eigenvalue problems of the $p$-Laplacian and that of 
the Laplacian}

\begin{thm}
Let $n \in \mathbb{N}$ and $p \in (1,\infty)$.
Let $u$ be an eigenfunction with $(n-1)$-zeros in $(0,L)$ 
for an eigenvalue $\lambda>0$ of the eigenvalue problem
\begin{equation}
\label{eq:dirichlet}
-(|u'|^{p-2}u')'=\lambda |u|^{p^*-2}u, \quad u(0)=u(L)=0,
\end{equation}
and $v$ an eigenfunction with $n$-zeros in $(0,L)$ 
for an eigenvalue $\mu>0$ of the eigenvalue problem
\begin{equation}
\label{eq:neumann}
-(|v'|^{p-2}v')'=\mu |v|^{p^*-2}v, \quad v'(0)=v'(L)=0.
\end{equation}
Then, the product $w=uv$ is an 
eigenfunction for the eigenvalue $\xi=2p^* (\lambda \mu)^{1/p}$ 
with $(2n-1)$-zeros in $(0,L)$ of the eigenvalue problem
\begin{equation}
\label{eq:laplace}
-w''=\xi |w|^{p^*-2}w,\quad w(0)=w(L)=0.
\end{equation}
\end{thm}

\begin{proof}
By \eqref{eq:ev} and \eqref{eq:ef}, 
the solution $(\lambda,u)$
of \eqref{eq:dirichlet} can be expressed
as follows:
\begin{gather*}
\lambda=\left(\frac{n\pi_{p,p^*}}{L}\right)^p|R|^{p-p^*},\\
u(x)=R\sin_{p,p^*}\left(\frac{n\pi_{p,p^*}}{L}x\right), \quad R \neq 0.
\end{gather*}
Similarly, by the symmetry \eqref{eq:sym},
the solution $(\mu,v)$ of \eqref{eq:neumann} 
is represented as 
\begin{gather*}
\mu=\left(\frac{n\pi_{p,p^*}}{L}\right)^p|Q|^{p-p^*},\\
v(x)=Q\left|\cos_{p,p^*}\left(\frac{n\pi_{p,p^*}}{L}x\right)\right|^{p-2}
\cos_{p,p^*}\left(\frac{n\pi_{p,p^*}}{L}x\right), \quad Q \neq 0.
\end{gather*}
Applying \eqref{eq:formula2} in Theorem \ref{thm:main}
and \eqref{eq:key}
to the product $w=uv$, we have
\begin{align*}
w(x)
&=RQ\sin_{p,p^*}\left(\frac{n\pi_{p,p^*}}{L}x\right)
\left|\cos_{p,p^*}\left(\frac{n\pi_{p,p^*}}{L}x\right)\right|^{p-2}
\cos_{p,p^*}\left(\frac{n\pi_{p,p^*}}{L}x\right)\\
&=2^{-2/p^*}RQ\sin_{2,p^*}\left(\frac{2n\pi_{2,p^*}}{L}x\right),
\end{align*}
which belongs to $C^2(\mathbb{R})$ and has $(2n-1)$-zeros in $(0,L)$. 
Therefore, by \eqref{eq:diffsin} with $p=2$,
a direct calculation shows
\begin{align}
\label{eq:w''}
w''&=-p^*2^{1-2/p^*} \left(\frac{n\pi_{2,p^*}}{L}\right)^2 RQ
\left|\sin_{2,p^*}\left(\frac{n\pi_{2,p^*}}{L}x\right)\right|^{p^*-2}
\sin_{2,p^*}\left(\frac{n\pi_{2,p^*}}{L}x\right) \notag \\
&=-p^*2^{3-4/p^*}\left(\frac{n\pi_{2,p^*}}{L}\right)^2
|RQ|^{2-p^*}|w|^{p^*-2}w.
\end{align}
On the other hand, \eqref{eq:key} gives  
\begin{equation}
\label{eq:lm}
(\lambda\mu)^{1/p}
=2^{2-4/p^*}\left(\frac{n\pi_{2,p^*}}{L}\right)^2|RQ|^{2-p^*}.
\end{equation}
Combining \eqref{eq:w''} and \eqref{eq:lm}, we obtain \eqref{eq:laplace}.
\end{proof}

\subsection{A pendulum-type equation}

We give an expression of the solution of the following initial value problem: 
\begin{equation}
\label{eq:p-pendulum}
-(|\theta'|^{p-2}\theta')'
=\lambda^p |\sin_{2,p}{\theta}|^{p-2}\sin_{2,p}{\theta}, \quad
\theta(0)=0, \ \theta'(0)=\omega_0.
\end{equation}

For $p,\ q \in (1,\infty)$ and $k \in [0,1)$ we define
$\am_{p,q}{(x,k)}$ by the inverse function of
$$\am_{p,q}^{-1}{(x,k)}:=\int_0^x \frac{d\theta}{(1-k^q|\sin_{p,q}{\theta}|^q)^{1/p^*}},
\quad -\infty<x<\infty,$$
in particular,
\begin{equation}
\label{eq:kpq}
K_{p,q}(k):=\am_{p,q}^{-1}\left(\frac{\pi_{p,q}}{2},k\right)
=\int_0^{\pi_{p,q}/2} \frac{d\theta}{(1-k^q\sin_{p,q}^q{\theta})^{1/p^*}}.
\end{equation}
Then, the function
$$\sn_{p,q}{(x,k)}:=\sin_{p,q}{(\am_{p,q}{(x,k)})}$$
is a $4K_{p,q}(k)$-periodic odd function in $\mathbb{R}$.
In case $p=q=2$, the functions $\am_{p,q},\ K_{p,q}$ and $\sn_{p,q}$
coincide with the amplitude function, the complete elliptic integral
of the first kind and the Jacobian elliptic function, respectively
(see \cite[\S 2.8]{PS} for them).
It is easy to see that 
$K_{p,q}(0)=\pi_{p,q}/2,\ 
\sn_{p,q}(x,0)=\sin_{p,q}{x},\ 
\lim_{k \to 1-0}K_{p,q}(k)=\infty$
and
$\lim_{k \to 1-0}\sn_{p,q}(x,k)=\tanh_q{x}$,
defined as the inverse function of
$$\tanh_q^{-1}{x}=\int_0^x \frac{dt}{1-|t|^q},\quad -1<x<1.$$
See \cite{T2} for details of the generalized Jacobian elliptic function.
Another type of generalization of the Jacobian elliptic function
also appears in a bifurcation problem of the $p$-Laplacian; see \cite{T}.

\begin{thm}
Let $p \in (1,\infty)$ and $\lambda,\ \omega_0 \in (0,\infty)$,
and set $k:=\omega_0/(2^{2/p} \lambda)$.
Then, the solution of \eqref{eq:p-pendulum}
is as follows: 
\begin{enumerate}
\item Case $k<1$.
\begin{equation*}
\theta (t)=2^{2/p} \sin_{p^*,p}^{-1}{\left(k
\sn_{p,p} \left(\lambda t,k\right)\right)},
\end{equation*}
which is a $4K_{p,p}(k)/\lambda$-periodic function and
$$\max_{0 \leq t<\infty}|\theta (t)|=2^{2/p} \sin_{p^*,p}^{-1}{k}.$$
\item Case $k=1$. 
\begin{equation*}
\theta (t)=2^{2/p} \sin_{p^*,p}^{-1}{(\tanh_p{(\lambda t)})},
\end{equation*}
which is a strictly monotone increasing function and
$$\lim_{t \to \infty}\theta (t)=2^{2/p-1}\pi_{p^*,p}=\pi_{2,p}.$$
\item Case $k>1$.
\begin{equation*}
\theta (t)=2^{2/p} \am_{p^*,p}{\left(k\lambda t,\frac1k \right)},
\end{equation*}
which is a strictly monotone increasing function and
$$\lim_{t \to \infty}\theta (t)=\infty.$$
\end{enumerate}
\end{thm}

\begin{proof}

By the standard argument (for example, \cite[Proposition 2.1]{DM} or
\cite[Theorem 3.1]{LE}), 
we can show that there exists a unique global
solution of \eqref{eq:p-pendulum}. 

Let $\theta$ be the solution of \eqref{eq:p-pendulum}
and $T:=\inf\{t:\theta'(t)=0\}$.
On the interval $(0,T)$, $\theta$ satisfies $\theta(t)>0$ and $\theta'(t)>0$.
Then, using \eqref{eq:diffsin} with $p=2$, we obtain
\begin{equation}
\label{eq:ee}
\frac1p \theta'(t)^p-\frac1p \omega_0^p
=\frac{2\lambda^p}{p} \cos_{2,p}\theta(t)-\frac{2\lambda^p}{p}.
\end{equation}
From \eqref{eq:formula} in Theorem \ref{thm:main} we have
\begin{equation}
\label{eq:mafc}
\cos_{2,p}\theta(t)
=1-2|\sin_{p^*,p}{(2^{-2/p}\theta(t))}|^p.
\end{equation}
Combining \eqref{eq:ee} and \eqref{eq:mafc} we obtain
$$\theta'(t)^p
=4\lambda^p(k^p-|\sin_{p^*,p}{(2^{-2/p}\theta(t))}|^p),$$
where $k=\omega_0/(2^{2/p} \lambda)$, hence, for $t \in [0,T]$,
\begin{equation}
\label{eq:t}
t=\frac{1}{2^{2/p} \lambda}
\int_0^{\theta(t)} \frac{d\theta}{(k^p-|\sin_{p^*,p}{(2^{-2/p}\theta)}|^p)^{1/p}}.
\end{equation}  

(i) Case $k<1$. 
We can find $\alpha \in (0,2^{2/p-1}\pi_{p^*,p})$ such that
$k=\sin_{p^*,p}(2^{-2/p}\alpha)$ and $T=\theta^{-1}{(\alpha)}$. 
Letting $\sin_{p^*,p}{(2^{-2/p}\theta)}=k\sin_{p,p}{\vp}$ 
in \eqref{eq:t}, we obtain
\begin{equation*}
t=\frac{1}{\lambda}
\int_0^{\vp(t)} \frac{d\vp}{(1-k^p|\sin_{p,p}{\vp}|^p)^{1/p^*}},
\end{equation*}
which implies
$$\vp(t)=\am_{p,p}(\lambda t,k).$$
Therefore,
$$\theta(t)=2^{2/p} \sin_{p^*,p}^{-1}{(k\sn_{p,p}(\lambda t,k))}.$$
We have thus found the unique solution $\theta$ of \eqref{eq:p-pendulum} in $[0,T]$.
However, in view of the periodicity properties of $\sn_{p,p}$, this function $\theta$
is actually the unique global solution of \eqref{eq:p-pendulum}, 
which is periodic of $4T=4K_{p,p}(k)/\lambda$ and 
whose maximum value is $\alpha=2^{2/p} \sin_{p^*,p}^{-1}{k}$.

(ii) Case $k=1$. In this case, letting $\sin_{p^*,p}(2^{-2/p}\theta)=x$
in \eqref{eq:t}, we obtain
$$t=\frac{1}{\lambda} \int_0^{x(t)} \frac{dx}{1-|x|^p}=\frac{1}{\lambda}\tanh_p^{-1}{x(t)}.$$
Therefore,
$$\theta(t)=2^{2/p} \sin_{p^*,p}^{-1}{(\tanh_p{(\lambda t)})}$$
and $T=\infty$. Moreover, by \eqref{eq:key}
$$\lim_{t \to \infty}\theta(t)
=2^{2/p} \sin_{p^*,p}^{-1}{1}=2^{2/p-1} \pi_{p^*,p}= \pi_{2,p}.$$

(iii) Case $k>1$. In this case, \eqref{eq:t} becomes
\begin{align*}
t
&=\frac{1}{2^{2/p} k \lambda}
\int_0^{\theta(t)} \frac{d\theta}{(1-k^{-p}|\sin_{p^*,p}{(2^{-2/p}\theta)}|^p)^{1/p}}\\
&=\frac{1}{k \lambda}
\int_0^{\vp(t)} \frac{d\vp}{(1-k^{-p}|\sin_{p^*,p}{\vp}|^p)^{1/p}}\\
&=\frac{1}{k \lambda}
\am_{p^*,p}^{-1}\left(\vp(t),\frac1k\right).
\end{align*}
Therefore, 
$$\theta(t)=2^{2/p} \am_{p^*,p}{\left(k\lambda t,\frac1k\right)}$$
and $T=\infty$. It is obvious that $\lim_{t \to \infty}\theta(t)=\infty$.
\end{proof}

\begin{rem}
The solution $\theta (t)$ in (ii) does not attain $\pi_{2,p}$ for any finite $t$,
while equations with $p$-Laplacian sometimes have flat-core 
solutions (cf. \cite{T}).
\end{rem}

\subsection{Catalan-type constants}

We define
$$G_p:=\sum_{n=0}^\infty \frac{(2/p)_n}{n!}\frac{(-1)^n}{(pn+1)^2}.$$
It is clear that $G_2=G$, i.e. Catalan's constant described in Introduction.

\begin{thm}
\label{thm:Gp}
Let $p \in (1,\infty)$, then
\begin{equation}
\label{eq:Gp}
G_p
=\frac{1}{2^{2/p}} \int_0^{\pi_{2,p}/2} \frac{x}{\sin_{2,p}{x}}\,dx
=\frac{1}{2^{2/p}} \int_0^1 K_{2,p}(k)\,dk,
\end{equation}
where $K_{2,p}(k)$ is defined by \eqref{eq:kpq}.
\end{thm}

\begin{proof}
By \eqref{eq:tan},
$$\int_0^1 \frac{\tau_{p^*,p}^{-1}{(x)}}{x}\,dx
=\sum_{n=0}^\infty \frac{(2/p)_n(-1)^n}{n!} \int_0^1 \frac{x^{pn}}{pn+1}\,dx
=G_p.$$
On the other hand, letting $\tau_{p^*,p}^{-1}{(x)}=2^{-2/p}y$,  we obtain
\begin{align*}
\int_0^1 \frac{\tau_{p^*,p}^{-1}{(x)}}{x}\,dx
&=\frac{1}{2^{4/p}}\int_0^{2^{2/p-2}\pi_{p^*,p}}
\frac{y}{\sin_{p^*,p}{(2^{-2/p}y)}\cos_{p^*,p}^{p^*-1}{(2^{-2/p}y)}}\,dy\\
&=\frac{1}{2^{2/p}} \int_0^{\pi_{2,p}/2} \frac{y}{\sin_{2,p}{y}}\,dy.
\end{align*}
Here, we have used \eqref{eq:key} and \eqref{eq:main1}.
This shows the first equality in \eqref{eq:Gp}.

The second equality in \eqref{eq:Gp} follows from Fubini's theorem that
\begin{align*}
\int_0^1 K_{2,p}(k)\,dk
&=\int_0^{\pi_{2,p}/2} \int_0^1
\frac{1}{(1-k^p\sin_{2,p}^p{x})^{1/2}}\,dk\,dx\\
&=\int_0^{\pi_{2,p}/2} \frac{1}{\sin_{2,p}{x}}
\int_0^{\sin_{2,p}{x}} \frac{1}{(1-t^p)^{1/2}}\, dt\, dx\\
&=\int_0^{\pi_{2,p}/2} \frac{\sin_{2,p}^{-1}{(\sin_{2,p}{x})}}
{\sin_{2,p}{x}}\,dx\\
&=\int_0^{\pi_{2,p}/2} \frac{x}
{\sin_{2,p}{x}}\,dx.
\end{align*}
The proof is accomplished.
\end{proof}

By Theorem \ref{thm:Gp} with $p=4$, we obtain

\begin{cor}
$$\frac{1}{\sqrt{2}} \int_0^{\varpi/2} \frac{x}{\slem{x}}\,dx
=\sum_{n=0}^\infty \frac{(1/2)_n}{n!} \frac{(-1)^n}{(4n+1)^2}.$$
\end{cor}

\begin{rem}
In a similar way to \cite[\S 3.2]{BE} or the last paragraph of \cite[\S 2.1]{LE}, 
we can also obtain the formula
\begin{equation}
\label{eq:Cp}
\int_0^{\pi_{2,p}/2} \frac{x}{\sin_{2,p}{x}}\,dx
=\frac{\pi_{2,p}}{2}\sum_{n=0}^\infty
\frac{(1/2)_n (1/p)_n}{(1/2+1/p)_n n!} \frac{1}{pn+1}=:\frac{\pi_{2,p}}{2}C_p.
\end{equation}
Therefore, from \eqref{eq:Gp}, \eqref{eq:Cp} and \eqref{eq:key},
\begin{equation}
\label{eq:pip*}
\frac{\pi_{p^*,p}}{4}=\frac{\pi_{2,p}}{2^{2/p+1}}
=\frac{G_p}{C_p}
=\frac{\sum_{n=0}^\infty \frac{(2/p)_n}{n!}\frac{(-1)^n}{(pn+1)^2}}
{\sum_{n=0}^\infty \frac{(1/2)_n (1/p)_n}{(1/2+1/p)_n n!} \frac{1}{pn+1}},
\end{equation}
particulary,
$$\frac{\pi}{4}=\frac{\sum_{n=0}^\infty \frac{(-1)^n}{(2n+1)^2}}
{\sum_{n=0}^\infty \left(\frac{(1/2)_n}{n!}\right)^2 \frac{1}{2n+1}}.$$
\end{rem}

\subsection{Series expansions of the lemniscate constant $\varpi$}

The series of Proposition \ref{prop:gls} for $p=2$ is nothing but the Gregory-Leibniz series.
Letting $p=4$ and using \eqref{eq:key}, we have the expansion series 
for the lemniscate constant $\varpi$:
\begin{equation}
\label{eq:lpi1}
\frac{\varpi}{2\sqrt{2}}
=1-\frac{1}{10}+\frac{1}{24}-\frac{5}{208}+\cdots
+\frac{(2n-1)!!}{(2n)!!}\frac{(-1)^n}{4n+1}+\cdots.
\end{equation}
On the other hand, there is the similar series to this in 
\cite[Theorem 5]{To}:
\begin{equation}
\label{eq:lpi2}
\frac{\varpi}{2}
=1+\frac{1}{10}+\frac{1}{24}+\frac{5}{208}+\cdots
+\frac{(2n-1)!!}{(2n)!!}\frac{1}{4n+1}+\cdots.
\end{equation}
Combining \eqref{eq:lpi1} and \eqref{eq:lpi2}, we obtain
\begin{align*}
\frac{2+\sqrt{2}}{8}\varpi
&=1+\frac{1}{24}+\cdots
+\frac{(4n-1)!!}{(4n)!!}\frac{1}{8n+1}+\cdots,\\
\frac{2-\sqrt{2}}{8}\varpi
&=\frac{1}{10}+\frac{5}{208}+\cdots
+\frac{(4n+1)!!}{(4n+2)!!}\frac{1}{8n+5}+\cdots.
\end{align*}
Finally, letting $p=4$ in \eqref{eq:pip*} we obtain
$$\frac{\varpi}{2\sqrt{2}}=
\frac{\sum_{n=0}^\infty \frac{(1/2)_n}{n!}\frac{(-1)^n}{(4n+1)^2}}
{\sum_{n=0}^\infty \frac{(1/2)_n (1/4)_n}{(3/4)_n n!} \frac{1}{4n+1}}.$$

\end{document}